\documentclass[english]{amsart}

\usepackage[latin1]{inputenc}

\usepackage{graphicx}

\usepackage[round]{natbib}
\usepackage{amsthm,amsmath,amssymb}
\theoremstyle{plain}
\newtheorem{thm}{Theorem}[section]
\newtheorem{prop}[thm]{Proposition}
\newtheorem{lem}[thm]{Lemma}

\newtheorem{cor}[thm]{Corollary}
\theoremstyle{definition}
\newtheorem*{defn}{Definition}

\theoremstyle{remark}
\newtheorem*{remark}{Remark}

\newcommand{\Fer}{\mathrm{Fer}}
\newcommand{\Gal}{\mathrm{Gal}}
\newcommand{\eqdef}{\stackrel{\rm def}{=}}
\newcommand{\xk}{X_1,\ldots,X_k}

\newcommand{\lk}{\lambda^{(1)},\ldots,\lambda^{(k)}}
\usepackage{babel}

\newcommand{\Des}{\mathrm{Des}}
\newcommand{\HDes}{\mathrm{HDes}}
\newcommand{\Par}{\mathrm{Par}}

\newcommand{\Hilb}{\mathrm{Hilb}}
\newcommand{\St}{\mathcal {ST}}

\newcommand{\qc}{G(r,p,q,n)}
\newcommand{\mc}{\mathbb C}
\newcommand{\fmaj}{\mathrm{fmaj}}
\newcommand{\Irr}{\mathrm{Irr}}
\newcommand{\GCD}{\mathrm{GCD}}

\author{Fabrizio Caselli}
\title[Projective reflection groups]{Projective reflection groups}
\address{Dipartimento di matematica, Universit\`a di Bologna, \\Piazza di Porta San Donato 5, \\ Bologna 40126, Italy}

\begin{document}

\maketitle
\begin{abstract}
We introduce the class of projective reflection groups which includes all complex reflection groups. We show that several aspects involving the combinatorics and the representation theory of all non exceptional irreducible complex reflection groups find a natural description in this wider setting.

\end{abstract}

\section{Introduction}
A complex reflection (or simply a reflection) is an endomorphism of a complex vector space $V$ which is of finite order and such that its fixed point space is of codimension 1. Finite reflection groups are finite subgroups of $GL(V)$ generated by reflections. They have probably been introduced by Shephard in \cite{Sh} and have been characterized by means of their ring of invariants and completely classified by Chevalley \cite{C} and Shephard-Todd \cite{ST} in the fifties, generalizing the well-known fundamental theorem of symmetric functions. In this classification there is an infinite family $G(r,p,n)$ of irreducible reflection groups, where $r,p,n$ are positive integers (with $r\equiv0 \mod p$) and 34 other exceptional groups. The relationship between the combinatorics and the (invariant) representation theory of symmetric groups is fascinating from both combinatorial and algebraic points of view, and the problem of generalizing these sort of results to all reflection groups has been faced in many ways. Besides several results that holds in the full generality of reflection groups, there are some relevant generalizations which have been obtained only for the wreath product groups $G(r,n)=G(r,1,n)$ (see, e.g., \cite{SW, St, AR, BRS, APR}). Some attempts to extend these results to other reflection groups have been made, in particular for Weyl groups of type $D$, (see, e.g., \cite{BC,BC1,BB}) though they are probably not completely satisfactory as in the case of wreath products.

 In this work we introduce a new class of groups, the projective reflection groups, which are a generalization of reflection groups. We first prove some results on the invariant theory of these groups that hold in full generality. Then we will focus our attention on the infinite family $G(r,p,q,n)$ of projective reflection groups, which includes all the groups $G(r,p,n)$ (in fact $G(r,p,1,n)=G(r,p,n)$). Fundamental in the theory of these groups is the following notion of duality: if $G=G(r,p,q,n)$ then we denote by $G^*=G(r,q,p,n)$ (where the roles of $p$ and $q$ have been interchanged). We note in particular that reflection groups $G$ satisfying $G=G^*$ are exactly the wreath products $G(r,n)=G(r,1,1,n)$ and that in general if $G$ is a reflection group then $G^*$ is not. We show that much of the theory of reflection groups can be extended to projective reflection groups and that the combinatorics of a projective reflection group $G$ of the form $\qc$ is strictly related to the (invariant) representation theory of $G^*$, generalizing several known results for wreath products in a very natural way.

The paper is organized as follows. We present definitions and a characterization in terms of invariants of projective reflection groups in \S \ref{maindef}. In \S \ref{coal} we further consider the action of a projective reflection group on a ring of polynomials to define and study its coinvariant algebra. In \S \ref{grpqn} we introduce the groups $G(r,p,q,n)$ and we answer some natural questions about possible isomorphisms between these groups. We exploit those combinatorial aspects of the groups $G(r,p,q,n)$ that we need and we describe a monomial basis for the coinvariant algebra in \S \ref{stat}. In \S \ref{irrep} we analyze the structure of the irreducible representations of these groups and we provide a combinatorial interpretation for their dimensions. In \S \ref{desrep} we consider a decomposition of the homogeneous components of the coinvariant algebra that leads us to define the descent representations of a projective reflection group: these representations are used in \S \ref{tendia} to describe the main new results of this paper. Here we show an explicit basis of the diagonal invariant algebra as a free module over the tensorial invariant algebra of all projective reflection groups $G(r,p,q,n)$. It is in this description that the interplay between a group $G$  and its dual $G^*$ attains its apex. In \S \ref{krocoe} we deduce some properties of Kronecker coefficients that can be deduce from the main results. In \S \ref{robshe} we extend the Robinson-Schensted correspondence on wreath products to all projective reflection groups of the form $G(r,p,q,n)$: in this general context it is not a bijection but it will be the key point to solve in \S \ref{gal} a problem posed by Barcelo, Reiner and Stanton on the Hilbert series of a certain diagonal invariant module twisted by a Galois automorphism. 

\section{Definitions and characterizations of projective reflection groups}\label{maindef}
Let $V$ be a finite dimensional complex vector space and consider the natural map $\varphi:GL(V)\rightarrow GL(S^q(V))$, where $S^q(V)$ is the $q$-th symmetric power of $V$. We clearly have $\ker \varphi=C_q$, where $C_q$ is the cyclic group of scalar matrices of order $q$ generated by $\zeta_qI$, with $\zeta_q\eqdef e^{\frac{2\pi i}{q}}$.\\
Now, if $W\subseteq GL(V)$ is a finite reflection group we have $\varphi (W)\cong W/(W\cap C_q)$. In particular, if $C_q\subset W$ we have that $W/C_q$ can be identified with a subgroup of $GL(S^q(V))$ by means of the map $\varphi$.
\begin{defn}
Let $G$ be a finite subgroup of $GL(S^q(V))$. We say that the pair $(G,q)$ is a \emph{projective reflection group} if there exists a reflection group $W\subset GL(V)$ such that $C_q\subseteq W$ and $G=W/C_q$.
\end{defn}
We mind the reader that we will usually omit the parameter $q$ whenever it can be clearly deduced from the context and we will simply talk about the projective reflection group $G$.
We also note that we obtain standard complex reflection groups in the case $q=1$.

Any finite subgroup of $GL(V)$ acts naturally on the symmetric algebra $S[V^*]$. 
A well-known theorem due to Chevalley \cite{C} and Shephard-Todd \cite{ST} says that a finite group $G$ of $GL(V)$ is a reflection group if and only if its invariant ring $S[V^*]^G$ is itself a polynomial algebra. Our next target is to generalize this result to the present context. We recall that a projective reflection group $(G,q)$ is equipped with an action on the symmetric power $S^q(V)$. The dual action can be extended to $S_q[V^*]$, the algebra of polynomial functions on $V$ generated by homogeneous polynomial functions of degree $q$.
The following result is the natural generalization of the theorem of Chevalley and Shephard-Todd to the present context.
\begin{thm}\label{charac}
Let $V$ be a complex vector space, $n=\dim V$, and $G$ be a finite group of graded automorphisms of $S_q[V^*]$, the algebra generated by homogeneous polynomial functions on $V$ of degree $q$. Then $(G,q)$ is a projective reflection group if and only if the invariant algebra $S_q[V^*]^G$ is generated by $n$ algebraically independent homogeneous elements.
\end{thm}
\begin{proof}
Let $(G,q)$ be a projective reflection group acting on $S^q(V)$ and $W\in GL(V)$ be a reflection group such that $G\cong W/C_q$. Then, since $C_q\subseteq W$, we have that $S[V^*]^W\subseteq S_q[V^*]$. It follows that $S[V^*]^W=S_q[V^*]^G$ and so $S_q[V^*]^G$ is generated by $n$ algebraically independent homogeneous elements by the corresponding result for reflection groups.
To prove the converse we need the following two lemmas.
\begin{lem}\label{qpo}
Let $\varphi$ be a graded automorphism of $S_q[V^*]$. Then for any $x\in V^*$ there exists $z\in V^*$ such that $\varphi (x^q)=z^q$. Moreover, for any $y\in V^*$ we have that $z^{i}$ divides $\varphi(x^i y^{q-i})$.
\end{lem}
\begin{proof}
If $\dim V=1$ then $\varphi(x^q)=\lambda x^q=(\lambda^{1/q}x)^q$ for
some $\lambda\in \mc$, where $\lambda^{1/q}$ is a $q$-th root of
$\lambda$. If $\dim V\geq 2$ let $y\in V^*$ be such that $x$ and $y$
are linearly independent. For $i=0,\ldots,q$ we let
$F_i=x^{i}y^{q-i}$. We observe that for all $i=2,\ldots,q $ we have
$F_{i-1}^2=F_{i-2}F_{i}$. So we also have 
\begin{equation}\label{eq}
 \varphi(F_{i-1})^2=\varphi(F_{i-2})\varphi(F_{i}).
\end{equation}
Since $\varphi$ is an automorphism we have that $\varphi(F_0)$ and $\varphi(F_q)$ are linearly independent. It follows that there exists an irreducible polynomial function $w$ which divides $\varphi(F_0)$ with multiplicity $m_0$ and $\varphi(F_q)$ with multiplicity $m_q$, with $m_0<m_q$. \\
We let $m_i$, with $i=0,\ldots,q$, be the multiplicity of $w$ as a factor of $\varphi(F_i)$ and we claim that $m_i=im_1-(i-1)m_0$ for all $i=0,\ldots,q$. This is clear if $i=0,1$. If $i>1$ then, by \eqref{eq}, we have $2m_{i-1}=m_{i-2}+m_{i}$. So $m_i=2m_{i-1}-m_{i-2}$ and, by a recursive argument, we conclude the proof of our claim. \\
If $m_{1}\leq m_0$ we have that $m_q=qm_{1}-(q-1)m_0\leq
qm_0-(q-1)m_0=m_0$ which contradicts our assumption $m_0<m_q$. So we
have $m_{1}\geq m_0+1$ and hence $$m_i=im_{1} -(i-1)m_0\geq i(m_0+1)
-(i-1)m_0=i+m_0\geq i.$$ It follows that $m_q=q$ and so
$F(x^q)=\lambda w^q=z^q$, where $z=\lambda^{1/q}w$.
\end{proof}
The following result is probably already known but we offer a proof for lack of a suitable reference.
\begin{lem}\label{qgen}
The algebra $S_q[V^*]$ is generated by the $q$-th powers of elements in $V^*$.
\end{lem}
\begin{proof}
We proceed by double induction on $q$ and $n=\dim V$. If $q=1$ or $n=1$ the result is trivial. So let $n=2$. If $\{x,y\}$ is a basis for $V^*$ then it is enough to show that every monomial $x^iy^{q-i}$ is a linear combination of $q$-th powers. So let $f_i=\binom{q}{i}x^iy^{q-i}$. Then for any $k=1,\ldots,q+1$ we have
$$
(kx+y)^q=\sum_{i=0}^q k^i f_i.
$$
Since the matrix of the coefficients is a Vandermonde matrix (and so it is invertible) we can express the $f_i$'s as a linear combination of the $(kx+y)^q$'s and we are done.
Now let $n>2$ and $M\in S_q[V^*]$ be of degree $q$. By induction on $q$ we can express $M$ as a linear combination of elements of the form $xy^{q-1}$, for some $x,y\in V^*$. But $xy^{q-1}$ belongs to $S^q(<x,y>)$ and the result follows from the case $n=2$.
\end{proof}

\noindent \emph{Proof of Theorem \ref{charac} (cont.)}. Let $x_0$ be any fixed non-zero element in $V^*$ and $g\in G$. By Lemma \ref{qpo} we have that $g(x_0^q)=z^q$ for some $z\in V^*$. Then, for any $y\in V^*$ we let
$$f(y)\eqdef g(x_0^{q-1}y)/z^{q-1}.$$
Note that this is well-defined by Lemma \ref{qpo} and that $f\in GL(V^*)$. We claim that for any $y\in V^*$ we have $g(y^q)=(f(y))^q$. By (the proof of) Lemma \ref{qpo} we know that $g((x_0+y)^q)$ belongs to the $q$-th symmetric power of a 2-dimensional subvector space of $V^*$ generated by $z$ and another element $w\in V^*$ such that $g(x_0^iy^{q-i})$ is a scalar multiple of $z^i w^{q-i}$. Lemma \ref{qpo} also tells us that $g((x_0+y)^q)$ is a $q$-th power of an element of the form $\lambda z+\mu w$. So we have
\begin{eqnarray*}
 g((x_0+y)^q) &=& g(x_0^q)+qg(x_0^{q-1}y)+\ldots+g(y^q)\\
 &=&\lambda^qz^q+q\lambda^{q-1}z^{q-1}\mu w+\ldots+\mu^q w^q
\end{eqnarray*}
Since the elements $z^i w^{q-i}$ are linearly independent we deduce that the corresponding elements in the previous equality must be the same. In particular we have $g(x_0^q)=\lambda^q z^q$, $g(x_0^{q-1} y)=\lambda^{q-1}z^{q-1}\mu w$ and $g(y^q)=\mu^q w^q$. By the first equality we have $\lambda^q=1$, by the second equality we have $f(y)=\lambda^{q-1}\mu w$ and so by the third equality we deduce that $(f(y))^q=(\lambda^{q-1}\mu w)^q=\mu^q w^q=g(y^q)$.
From this and Lemma \ref{qgen} we deduce that $g(x_1\cdots x_q)=f(x_1)\cdots f(x_q)$ for all $x_1,\ldots,x_q\in V^*$. We observe that  in the previous formula we can substitute $f$ with its scalar multiple $\zeta_q^i f$ and we denote by $W_g=\{l\in GL(V^*):g(v_1\cdots v_q)=l(v_1)\cdots l(v_q)\textrm{ for all }v_1,\ldots,v_q\in V^*\}$. It is clear from our discussion that $|W_g|=q$ (being a coset of $C_q$ in $GL(V^*)$) and that $W=\cup W_g$ is a subgroup of $GL(V^*)$. We also observe that $W_1=C_q$. This implies that $S[V^*]^W\subset S_q[V^*]$. It follows that $S[V^*]^W=S_q[V^*]^G$. By the theorem of Chevalley and Shephard-Todd we deduce that $W$ is a reflection group and so, by definition, $(G,q)$ is a projective reflection group since, clearly, $G=W/C_q$.
 
\end{proof}

\section{The coinvariant algebra of a projective reflection group}\label{coal}
If $W$ is a reflection group acting on a vector space $V$ then the coinvariant algebra of $W$ is the quotient of the ring of polynomial functions by the ideal generated by homogeneous invariant polynomial functions of positive degree. It is a rich and fundamental tool in the study of the invariant theory of $W$ (see, e.g., \cite[\S3]{H}). As we observed in \S \ref{maindef} we have an action of any projective
reflection group $(G,q)$ on the algebra $S_q[V^*]$ generated by the homogeneous polynomial functions on $V$ of degree $q$. We denote by $I_+^G$ the ideal of $S_q[V^*]$ generated by homogeneous elements in $S_q[V^*]^G$ of positive degree, and define the \emph{coinvariant algebra} of $G$ by
$$
R^G\eqdef S_q[V^*]/I_+^G.
$$


The coinvariant algebra $R^G$ is a graded representation of $G$ and, as it was the case for reflection groups, it is isomorphic as a $G$-module to the group algebra $\mathbb C G$.
\begin{prop}\label{reg}
If $G$ is any projective reflection group, we have an isomorphism of $G$-modules $$R^G\cong \mc G.$$
\end{prop}
\begin{proof}
Let $W\in GL(V)$ be a reflection group containing $C_q$ such that
$G=W/C_q$ and recall that the invariants of $W$ coincide with the
invariants of $G$. It follows that the coinvariant algebra $R^G$ is the subalgebra of $R^W$
given by the elements of degree multiple of $q$. 
Now let $\phi$ be an irreducible representation of $G$. We have to
prove that $\phi$ appears in the irreducible decomposition of the
$G$-module $R^G$ with multiplicity equal to its dimension. The representation $\phi$ is
also an irreducible representation of $W$ and so it appears in the
coinvariant algebra of $R^W$, with multiplicity equal to its
dimension. Then, since $\phi$ is a representation of $G$ it must be
fixed by $C_q$ and so its isotypical component appears in the
homogeneous pieces of degree multiple of $q$, i.e. in $R^G$. The
result follows.
\end{proof}

\section{The groups $G(r,p,q,n)$}\label{grpqn}
In this section we introduce an important class of projective reflection groups: they are those projective reflection groups arising as quotients (by scalar subgroups) of all non exceptional irreducible complex reflection groups.\\
We know by the work of Shephard-Todd \cite{ST} that all but a finite number of irreducible reflection groups are the groups $G(r,p,n)$ that we are going to describe. If $A$ is a matrix with complex entries we denote by $|A|$ the real matrix whose entries are the absolute values of the entries of $A$.
The groups $G(r,n)= G(r,1,n)$ are given by all $n\times n$ matrices satysfying the following conditions: 
\begin{itemize}
\item the non-zero entries are $r$-th roots of unity;
\item there is exactly one non-zero entry in every row and every column (i.e. $|A|$ is a permutation matrix).
\end{itemize}
If $p$ divides $r$ then the reflection group $G(r,p,n)$ is the subgroup of $G(r,n)$ given by all matrices $A\in G(r,n)$ such that $\frac{\det A}{\det |A|}$ is a $\frac{r}{p}$-th root of unity. 

It is easy to characterize all possible scalar subgroups of the groups $G(r,p,n)$: in fact we can easily observe that the scalar matrix $\zeta_qI$ belongs to $G(r,p,n)$ if and only if $q|r$ and $pq|rn$.

\begin{defn}
Let $r,p,q,n\in \mathbb N$ be such that $p|r$, $q|r$ and $pq|rn$. Then we let
$$
G(r,p,q,n)\eqdef G(r,p,n)/C_q,
$$
where $C_q$ is the cyclic group generated by $\zeta_q I$.
\end{defn}
We observe that starting from the wreath product group $G(r,n)$ we
could have done first the quotient by the subgroup $C_q$ and then
taken the subgroup of this quotient consisting of elements $A$
satysfing $\frac{\det A}{\det |A|}$ is a $\frac{r}{p}$-th root of
unity (note that this requirement would have been well-defined). We
would have obtained the same group $G(r,p,q,n)$ and one of the targets
of this paper is to convince the reader that these two operations of
``taking subgroups'' and ``taking quotients'' have the same dignity
and for many aspects their are dual to each other. In fact, we may note
the symmetry on the conditions for the parameters $p$ and $q$ in the
definition of the group $\qc$. In particular if $G=(\qc,q)$ then the
projective reflection group $G^*\eqdef (G(r,q,p,n),p)$, where the roles of the parameters
$p$ and $q$ are interchanged, is always well-defined. The classical
Weyl groups of type $A$, $B$ and $D$ are respectively in this notation
the groups $G(1,1,1,n), G(2,1,1,n)$ and $G(2,2,1,n)$. Note that while
Weyl groups of type $A$ and $B$ are fixed by the $*$-operator, Weyl
groups of type $D$ and general reflection groups are not. The main
goal of this work is to show that several aspects of the invariant
theory of a projective reflection group $G$ of the form $(\qc,q)$ is
strongly related to and easily described by the combinatorics of
$G^*$. From now on, when talking about a projective reflection group
we always refer to a group of the form $(\qc,q)$ (and we will denote it
simply by $\qc$) unless otherwise specified.

If $g\in G(r,n)$ we write $g=[\sigma;c_1,\ldots,c_n]$ if the non-zero entry in the $i$-th row of $g$ is $\zeta_r^{c_i}$ and $\sigma\in S_n$ is the permutation associated to $|g|$ (i.e. $\sigma(i)=j$ if $g_{i,j}\neq 0$). We observe that $g$ determines $\sigma$ uniquely while the integers $c_i$ are determined only modulo $r$. We also note that in this notation we have that $g=[\sigma;c_1,\ldots,c_n]$ belongs to $G(r,p,n)$ if and only if $\sum c_i\equiv 0 \mod p$.\\
If $g\in \qc$ we also write $g=[\sigma;c_1,\ldots,c_n]$ to mean that $g$ can be represented by $[\sigma;c_1,\ldots,c_n]$ in $G(r,p,n)$ and we let $c(g)\eqdef \sum c_i$ (note that $c(g)$ is defined modulo $\GCD(r,\frac{rn}{q})$ which is a multiple of $p$).
One may ask for which choice of the parameters one has $G\cong G^*$ as abstract groups. This is the main target of this section.
\begin{lem}\label{represe}
Let $d=\GCD (\frac{rn}{q}, p')|p$ for some $p'|r$. Then every element in $\qc$ has exactly $\frac{qd}{p'}$ representatives in $G(r,p',n)$.
\end{lem}
\begin{proof}
Let $g\in \qc$ and $g_0,\ldots,g_{q-1}\in G(r,p,n)$ be the $q$ distinct representatives of $g$. After a suitable reordering of these elements we have that $c(g_i)=c(g_0)+\frac{rn}{q}i$. So we have to compute the number of solutions modulo $q$ of the equation
$$
c(g_0)+\frac{rn}{q}i\equiv 0 \mod p'.
$$
This equation has $d$ solutions modulo $p'$ and hence $\frac{r}{p'}d$ solutions modulo $r$ and so $\frac{qd}{p'}$ solutions modulo $q$ (note that we know a priori that the solutions are defined modulo $q$).

\end{proof}
\begin{prop}\label{suff}
Let $\qc$ and $G(r,p',q',n)$ have the same cardinality (i.e. $pq=p'q'$) and assume that $\GCD (\frac{rn}{q},p')=\GCD(\frac{rn}{q'},p)$. Then
$$
\qc\cong G(r,p',q',n). 
$$
\end{prop}
\begin{proof}
Let $d= \GCD (\frac{rn}{q},p')=\GCD(\frac{rn}{q'},p)$. Then we clearly have that $d|p,p'$ and so we can apply Lemma \ref{represe} to both $\qc$ and $G(r,p',q',n)$. Observe that $\frac{qd}{p'}=\frac{q'd}{p}$ since $pq=p'q'$. We let $\tilde p=\textrm{lcm}(p,p')$ and $\tilde q=\frac{qd}{p'}=\frac{q'd}{p}$. Then we can easily check that the group $G(r,\tilde p,\tilde q,n)$ is well-defined and that $\tilde p\tilde q=pq$. By Lemma \ref{represe} there exists a unique map
$\varphi:G(r,p,q,n)\longrightarrow G(r,\tilde p,\tilde q,n)$ such that $g$ and $\varphi(g)$ have a common representative in $G(r,n)$. We observe that this map is injective by construction and that it is a group homomorphism since we can perform the group operations by means of the common representatives. We conclude that it is a group isomorphism since the two groups have the same cardinality. We can similarly prove that $G(r,p',q',n)\equiv G(r,\tilde p, \tilde q,n)$ and the proof is complete.
\end{proof}

\begin{prop}\label{dec}
Let $\GCD(\frac{rn}{pq},\frac{r}{q})=\GCD(\frac{rn}{pq},\frac{r}{p})=\delta_1\delta_2$, where $\delta_1$ is a product of primes which appear with the same multiplicity in the prime factorizations of $p$ and $q$, and $\delta_2$ is a product of primes which appear with different multiplicities in the prime factorizations of $p$ and $q$. Then $$\qc\cong G(r,\delta_2p,q,n)\times C_{\delta_2}.$$
\end{prop}
\begin{proof}
Observe that $G(r,\delta_2p,q,n)$ is well-defined (i.e. $\delta_2p|r$ and $\delta_2pq|nr$) and  that it is a
normal subgroup of $\qc$. The subgroup $C_{\delta_2}$ is the subgroup
generated by the scalar element $\zeta_{{\delta_2q}}$ which has order
$\delta_2$ in $\qc$. We only have to show that the two subgroups
$G(r,\delta_2p,q,n)$ and $C_{\delta_2}$ intersect trivially in
$\qc$. For this we have to verify that $\zeta_{{\delta_2q}}^k\in
G(r,\delta_2p,q,n)$ if and only if $k\equiv 0  \mod \delta_2$. With this in mind we observe that $\GCD(\frac{rn}{pq\delta_2},\delta_2)=1$. In fact
if $\pi$ is a prime such that $\pi^a||\delta_2$ (i.e. $\pi^a|\delta_2$
and $\pi^{a+1}\not |\delta_2$) then $\pi$ appears with different multiplicities in the prime factorizations of $p$ and $q$ and hence also in the prime factorizations of $\frac{r}{p}$ and $\frac{r}{q}$. Then, since $\GCD(\frac{rn}{pq},\frac{r}{q})=\GCD(\frac{rn}{pq},\frac{r}{p})$, we necessarily have that $\pi^a||\frac{rn}{pq}$ and so the condition $\GCD(\frac{rn}{pq\delta_2},\delta_2)=1$ follows. So
\begin{eqnarray*}
\zeta_{{\delta_2q}}^k\in G(r,\delta_2p,q,n)&\Leftrightarrow
&\frac{rn}{q\delta_2}k\equiv 0 \mod \delta_2 p\\
&\Leftrightarrow &\frac{rn}{pq\delta_2}k\equiv 0 \mod \delta_2\\
&\Leftrightarrow & k\equiv 0 \mod \delta_2.
\end{eqnarray*}
\end{proof}
\begin{thm}\label{isomo}
Let $G=\qc$, with $n\neq 2$. Then $$G\cong G^* \Longleftrightarrow \GCD \left(\frac{rn}{pq},\frac{r}{q}\right)=\GCD \left(\frac{rn}{pq},\frac{r}{p}\right).$$
\end{thm}
\begin{proof}
Assume that $\GCD (\frac{rn}{pq},\frac{r}{q})=\GCD(\frac{rn}{pq},\frac{r}{p})$. By Proposition \ref{dec} we have $G\cong G(r,\delta_2p,q,n)\times C_{\delta_2}$ and $G^*\cong G(r,\delta_2q,p,n)\times C_{\delta_2}$. So we only have to show that $G(r,\delta_2p,q,n)\cong G(r,\delta_2q,p,n)$. 
By Proposition \ref{suff} it is enough to verify that
$\GCD(\frac{rn}{q},\delta_2q)=\GCD(\frac{rn}{p},\delta_2p)$ or
equivalently
$\GCD(\frac{rn}{pq}p,\delta_2q)=\GCD(\frac{rn}{pq}q,\delta_2p)$. If
$\pi$ is a prime number with the same multiplicity in the prime
factorizations of $p$ and $q$ then  we clearly have that $\pi^a|\GCD(\frac{rn}{pq}p,\delta_2q)$ if and only if $\pi^a|\GCD(\frac{rn}{pq}q,\delta_2p)$. If $\pi$ appears with different multiplicities in the prime factorizations of $p$ and $q$ then, by the definition of $\delta_2$, we have $\pi^a|\delta_2$ if and only if $\pi^a|\GCD(\frac{rn}{pq},\frac{r}{p})=\GCD(\frac{rn}{pq},\frac{r}{q})$ if and only if $\pi^a|\frac{rn}{pq}$. So $\pi$ appears with the same multiplicity in the prime factorizations of $\frac{rn}{pq}$ and $\delta_2$. This implies that $\GCD(\frac{rn}{pq}p,\delta_2q)=\GCD(\frac{rn}{pq}q,\delta_2p)$ and we are done. 

Now assume that $G\cong G^*$. We recall that for $n>2$ the center of
the symmetric group $S_n$ is trivial. Then one can easily prove that the center of $G$ is given by the scalar elements only. The number of scalar elements in $G$ is exactly $\GCD(\frac{rn}{pq},\frac{r}{q})$ and the number of scalar elements in $G^*$ is $\GCD(\frac{rn}{pq},\frac{r}{p})$ and so the result follows. If $n=1$ the groups $G$ and $G^*$ are always cyclic groups of order $\frac{r}{pq}$ and so there is nothing to prove. 
\end{proof}
\begin{cor}\label{sca}
Let $G=\qc$ and $n\neq 2$. Then $G\cong G^*$ if and only if $G$ and $G^*$ have the same number of scalar elements.
\end{cor}
\begin{remark}
If $n=2$  the condition of having the same number of scalar elements is always sufficient for $G$ to be isomorphic to $G^*$ but it is no longer necessary. For example if $G=G(2,2,1,2)$ we have that $G$ has two scalar elements while $G^*$ has only one scalar element. Nevertheless we have that $G$ and $G^*$ are both isomorphic to the Klein group of order 4.
\end{remark}
As an application of Corollary \ref{sca} we can consider the Weyl group $G=G(2,2,1,n)=D_n$ and its dual group $G^*=G(2,1,2,n)=B_n/\pm 1$. We have that $G\cong G^*$ if and only if $n$ is odd. In fact $G$ has two scalars if $n$ is even and one scalar if $n$ is odd while $G^*$ has exactly one scalar for every value of $n$.


We recall to the reader that a projective reflection group is a pair
$(G,q)$ and so two projective reflection groups  should not be considered isomorphic  as long as the parameters $q$ are different. So Theorem \ref{isomo} and Corollary \ref{sca} should be rather seen as an answer to a natural question and not as a piece of information for the classification of projective reflection groups. Nevertheless, this result is important whenever we consider the representation theory of the abstract group $G$ independently of its action on the algebra of polynomials $S_q[V^*]$.
\section{Descent-type statistics}\label{stat}
The study of permutation statistics, and in particular of those depending on up-down or descents patterns, is a very classical subject of study in algebraic combinatorics that goes back to the early 20th century to the works of Macmahon \cite{M}, and has found a new interest in the last decades after the fundamental work of Adin and Roichman \cite{AR}. In this section we extend these concepts to our context and in particular we define some multivariate statistics on the projective reflection groups $\qc)$ that we need to describe some aspects of their invariant theory. 

For $g=[\sigma;c_1,\ldots,c_n]\in \qc$ we let
\begin{eqnarray*}
\HDes(g)&\eqdef&\{i\in [n-1]:\,c_i\equiv c_{i+1}\textrm{ and }\sigma_i>\sigma_{i+1}\}\\
h_i(g)&\eqdef&\#\{j\geq i:\,j\in \HDes(g)\}\\
k_i(g)&\eqdef& \left\{\begin{array}{ll}[c_n]_{r/q}& \textrm{if }i=n\\k_{i+1}+[c_i-c_{i+1}]_r& \textrm{if }i\in [n-1],
\end{array}\right.
\end{eqnarray*}
where $[c]_s$ is the smallest non negative representative of the class of the integer $c$ modulo $s$.\\
Note that these statistics do not depend on the choice of the integers $c_1,\ldots,c_n$ for representing $g$. We call the elements in  $\HDes(g)$ the \emph{homogeneous descents} of $g$.\\
For example, let $g=[27648153;2,3,3,5,1,7,3,2]\in G(6,2,3,8)$. Then $\HDes(g)=\{2,5\}$, \\$(h_1,\ldots,h_8)=(2,2,1,1,1,0,0,0)$ and $(k_1,\ldots,k_8)=(18,13,13,9,5,5,1,0)$.

If $q=p=1$ these statistics give an alternative definition for the flag-major index of Adin and Roichman (see \cite{AR}) for wreath products $G(r,n)$. In fact, if we let $d_i(g) \eqdef\#\{j\geq i:\,j\in \Des(g)\}$, where 
$$\Des(g)=\{i:\textrm{ either }[c_i]_r<[c_{i+1}]_r\textrm{ or }[c_i]_r=[c_{i+1}]_r \textrm{ and }\sigma_i>\sigma_{i+1}\},$$
then the \emph{flag-major index} is defined as
$$
\fmaj(g)\eqdef r\sum_{i\in \Des(g)}i+\sum_i[c_i]_r=\sum_i (r d_i(g)+[c_i]_r).
$$

\begin{lem}\label{oldnew}
Let $p=q=1$ and $g=[\sigma,c_1,\ldots,c_n]\in G(r,n)$. Then
$$
r\cdot h_i(g)+k_i(g)=r\cdot d_i(g)+[c_i]_r,
$$
and in particular $\sum(r\cdot h_i(g)+k_i(g))=\fmaj(g)$.
\end{lem}
\begin{proof}
We proceed by reverse induction on $i$. If $i=n$ it is clear, since $d_n=h_n=0$ and $k_n=[c_n]_r$ by definition. So let $i<n$. 
\begin{itemize}
\item If $[c_i]_r=[c_{i+1}]_r$ then $i\in \Des(g)$ if and only if $i\in \HDes(g)$ and the result easily follows.
\item If $[c_i]_r<[c_{i+1}]_r$ we have $d_i=d_{i+1}+1$, $h_i=h_{i+1}$ and $k_i=k_{i+1}+r+[c_i]_r-[c_{i+1}]_r$. So
$$
rh_i+k_i=rh_{i+1}+k_{i+1}+r+[c_i]_r-[c_{i+1}]_r=rd_{i+1}+[c_{i+1}]_r+r+[c_i]_r-[c_{i+1}]_r=rd_i+[c_i]_r.
$$
\item If $[c_i]_r>[c_{i+1}]_r$ we have $d_i=d_{i+1}$, $h_i=h_{i+1}$ and $k_i=k_{i+1}+[c_i]_r-[c_{i+1}]_r$ and the result follows similarly.
\end{itemize}
\end{proof}
We note that if $\lambda_i(g)\eqdef r\cdot h_i(g)+k_i(g)$ then the sequence $\lambda(g)\eqdef(\lambda_1(g),\ldots,\lambda_n(g))$ is a partition (actually both sequences $(h_1,\ldots,h_n)$ and $k_1,\ldots,k_n$ are partitions). We may also observe that $\lambda(g)$ is such that $g=[|g|;\lambda(g)]$. 
Extending the notion of flag-major index we define the flag-major index for the projective reflection group $G(r,p,q,n)$ by $\fmaj(g)\eqdef |\lambda(g)|$. 

As a first application of these definitions we can describe an explicit basis of the coinvariant algebra $R^G$ of a projective reflection group of the form $G=G(r,p,q,n)$. We already know that $\dim R^G=|G|$ and so it could be natural to expect a basis for the algebra $R^G$ indexed by elements of $G$. As it was mentioned in the introduction, this is the first occurrence of the invariant theory of a projective reflection group $G$ which is naturally described by its dual group $G^*$. 
Generalizing and unifying results and definitions in \cite{GS,ABR,BC,BB}, we associate to any element $g\in G$ a monomial $a_g\in \mc[X]\eqdef\mc[x_1,\ldots,x_n]$ of degree $\fmaj(g)$ in the following way 
$$a_g(X)\eqdef \prod_i x_{|g|(i)}^{\lambda_i(g)}.$$ 
We denote by $S_q[X]$ the algebra of polynomials in $\mc[X]$ generated by the monomials of degree $q$. 
\begin{lem}\label{deg}
Let $g\in G(r,1,q,n)$. Then the following are equivalent
\begin{itemize}
\item $g\in \qc$;
\item $\fmaj(g)\equiv 0 \mod p$;
\item $a_g\in S_p[X]$.
\end{itemize}

\end{lem}
\begin{proof}
We can choose an expression $g=[\sigma;c_1,\ldots,c_n]$ for $g$ such that $[c_n]_{r/q}=[c_n]_r$. So we can assume that $q=1$. Then, by Lemma \ref{oldnew}, $c_i\equiv k_i \mod r$ and so 
$$
\fmaj(g)=\deg a_g=\sum_i r h_i+k_i\equiv \sum_i k_i\equiv \sum c_i \mod r,
$$
and the result follows.
\end{proof}
\begin{thm} \label{coba}
If $G=G(r,p,q,n)$ then the set $\{a_g:\,g\in G^*\}$ represents a basis for $R^{G}$.
\end{thm}
\begin{proof}
This fact is implicitly proved in \cite{BB} if $q=1$. Now assume that $p=1$. In this case the coinvariant algebra $R^{G(r,1,q,n)}$ has a basis given by
$$
\{a_g:g\in G(r,1,1,n) \textrm{ and }\deg(a_g)\equiv 0 \mod q\}.
$$ By Lemma \ref{deg} this set is exactly $\{a_g:g\in
G(r,q,1,n)\}$. \\
Now consider the general case, i.e. $p,q$
arbitrary, and recall that $R^G$ is the subalgebra of $R^{G(r,p,1,n)}$
given by the elements of degree multiple of $q$.
If $g\in G^*$ then it is clear that there exists an element $g'$ which represents $g$ in $G(r,q,1,n)$ such that $\lambda_i(g)=\lambda_i(g')$ for all $i$ and so $a_g=a_{g'}$. So the elements $a_g$ with $g\in G^*$ are independent (since they belong to a basis of $R^{G(r,1,q,n)}$), and belong to $R^{G}$ by Lemma \ref{deg}. Since they have the right cardinality they form a basis for $R^{\qc}$.
\end{proof}
The following result, which is stronger than Theorem \ref{coba} is one of the main ingredients in the description of the descent representations (see \S \ref{desrep}). For this we need to introduce some notation on partitions. If $M$ is a monomial in $\mathbb C[X]$ we denote by $\lambda(M)$ its \emph{exponent partition}, i.e. the partition obtained by rearranging the exponents of $M$. We say that a polynomial is homogeneous of \emph{partition degree} $\lambda$ if it is a linear combination of monomials whose exponent partition is $\lambda$. We consider the set of partitions of a given integer $k$ as a partially ordered set by the dominance order. In this order we have $\lambda \unlhd \mu$ if $\lambda_1+\ldots+\lambda_i\leq \mu_1+\cdots+\mu_i$ for all $i$. \\
We denote by $S_q[X]_\lambda$ the space of homogeneous polynomials of partition degree $\lambda$ and we let 
$$S_q[X]_{\lhd \lambda}=\bigoplus_{\lambda'\lhd \lambda}S_q[X]_{\lambda'}\textrm{ and }  S_q[X]_{\unlhd \lambda}=\bigoplus_{\lambda'\unlhd \lambda}S_q[X]_{\lambda'}.$$
\begin{cor}\label{klo}
Let $G=\qc$ and  $M\in S_q[X]$ be a monomial and consider the unique expression
$$
M=\sum_{g\in G^*} f_ga_g
$$
with $f_g\in S_q[X]^{G}$. Then $f_g\in S_q[X]_{\unlhd \lambda}$ for some partition $\lambda$ such that $\lambda+\lambda(g)\unlhd
\lambda(M)$. In particular, if $\fmaj(g)=\deg M$ and $f_g\neq 0$ we
have $\lambda(g)\unlhd \lambda(M)$.
\end{cor}
\begin{proof}
This fact is known (see \cite{ABR,BC,BB}) if $q=1$. So in particular we know that there exist unique homogeneous polynomials $f_g\in S[X]^{G(r,p,1,n)}$, with $g\in G(r,1,p,n)$, such that
$$M=\sum_{g\in G(r,1,p,n)}f_ga_g,$$ and that these polynomials satisfy $f_g\in S[X]_{\unlhd \lambda}$ for some partition $\lambda$ such that $\lambda+\lambda(g)\unlhd \lambda(M)$. 
By Corollary \ref{coba} we know that there exists a unique expression
$$
M=\sum_{g\in G^*}f_g a_g,
$$
with $f_g\in S_q[X]^G$. Since the index set in the second sum is a subset of the index set in the first sum and $S_q[X]^G\subseteq S[X]^{G(r,p,1,n)}$ these two expressions coincide and the result follows.
\end{proof}
\section{Irreducible representations}\label{irrep}
In this section we describe explicitly a natural parametrization of the irreducible representations of a projective reflection group $\qc$. 
Given a partition $\mu$ of $n$, the \emph{Ferrers diagram of shape $\mu$} is a collection of boxes, arranged in left-justified rows, with $\mu_i$ boxes in row $i$. 
We denote by $\Fer(r,p,n)$ the set of $r$-tuples of Ferrers diagrams whose shapes $(\lambda^{(0)},\ldots,\lambda^{(r-1)})$ are such that $\sum |\lambda^{(i)}|=n$ and $\sum_i i|\lambda^{(i)}|\equiv 0 \mod p$. This may recall the definition of $G(r,p,n)$ where the role of $\sum_i c_i(g)$ is played by $\sum_i i|\lambda^{(i)}|$. In an extreme parallelism with the groups $G(r,p,n)$ we have the following result. 
\begin{lem}\label{cyc}
Let $(\lambda^{(0)},\ldots,\lambda^{(r-1)})\in \Fer(r,p,n)$ (and $q\in \mathbb N$ be such that  $q|r$ and $pq|rn$). Then
$$
(\lambda^{(r/q)},\lambda^{(r/q+1)}\ldots,\lambda^{(r/q+r-1)})\in \Fer(r,p,n),
$$ 
where $\lambda^{(j)}\eqdef \lambda^{(j-r)}$ if $j\geq r$.
\end{lem}
\begin{proof}
We have
$$
\sum_{i=0}^{r-1} i|\lambda^{(r/q+i)}|\equiv \sum_{i=0}^{r-1}(i-r/q)|\lambda^{(i)}|=\sum_{i=0}^{r-1}i|\lambda^{(i)}|-rn/q,
$$
and the result follows.
\end{proof}
If $\mu\in \Fer(r,p,n)$ we denote by $\St_{\mu}$ the set of all possible fillings of the boxes in $\mu$ with all the numbers from 1 to $n$ appearing once, in such way that rows are increasing from left to right and columns are incresing from top to bottom in every single Ferrers diagram of $\mu$. Moreover we let $\St(r,p,n)\eqdef \cup_{\mu\in \Fer(r,p,n)} \St_\mu$.\\
By Lemma \ref{cyc} we have a natural action of $C_q$ on both $\Fer(r,p,n)$ and  $\St(r,p,n)$. We denote the corresponding quotient sets by $\Fer(r,p,q,n)$ and $\St(r,p,q,n)$. 
If $T\in \St(r,p,q,n)$ we denote by $\mu(T)$ its corresponding shape in $\Fer(r,p,q,n)$ and if $\mu \in \Fer(r,p,q,n)$ we let $\St_\mu \eqdef \{T\in \St(r,p,q,n): \mu(T)=\mu\}$. Finally, if $\Fer=\Fer(r,p,q,n)$, we let $\Fer^*\eqdef \Fer(r,q,p,n)$.\\
We can define the statistics $h_i$ and $k_i$ in $\St(r,p,q,n)$
similarly to the case of $\qc$. Let $T\in \St(r,p,q,n)$ be represented
by $(T_1,\ldots,T_r)$. For $i\in [n]$ we let $c_i=j$ if $i\in
T_j$. With this notation in mind we let
\begin{eqnarray*}
\HDes(T)&\eqdef&\{i\in [n-1]:\, c_i=c_{i+1} \textrm{ and $i$ appears strictly above $i+1$}\};\\
h_i(T)&\eqdef&\#\{j\geq i:\,j\in \HDes(T)\};\\
k_i(T)&\eqdef& \left\{\begin{array}{ll}[c_n]_{r/q}& \textrm{if }i=n\\k_{i+1}+[c_i-c_{i+1}]_r& \textrm{if }i\in [n-1].
\end{array}\right.
\end{eqnarray*} 
It is clear that these definitions do not depend on the choice of the representative $(T_1,\ldots,T_r)$. For example if
\begin{center}
$T=\left( \begin{picture}(90,15)
\put(0,12){\line(1,0){20}}
\put(35,12){\line(1,0){20}}
\put(70,12){\line(1,0){20}}
\put(0,2){\line(1,0){20}}
\put(35,2){\line(1,0){20}}
\put(70,2){\line(1,0){20}}
\put(0,-8){\line(1,0){10}}
\put(35,-8){\line(1,0){20}}
\put(0,-8){\line(0,1){20}}
\put(10,-8){\line(0,1){20}}
\put(20,2){\line(0,1){10}}
\put(35,-8){\line(0,1){20}}
\put(45,-8){\line(0,1){20}}
\put(55,-8){\line(0,1){20}}
\put(70,2){\line(0,1){10}}
\put(80,2){\line(0,1){10}}
\put(90,2){\line(0,1){10}}
\put(3,-6){$5$}
\put(3,4){$1$}
\put(13,4){$4$}
\put(38,-6){$3$}
\put(38,4){$2$}
\put(48,-6){$9$}
\put(48,4){$8$}
\put(73,4){$6$}
\put(83,4){$7$}
\put(22,2){,}
\put(57,2){,}
\end{picture}\right)\in \St(3,1,3,9)$,
\end{center}
we have $(h_1,\ldots,h_9)=(3,3,2,2,1,1,1,1,0)$ and $(k_1,\ldots,k_9)=(5,3,3,2,2,1,1,0,0)$.
We define
$\lambda_i(T)\eqdef rh_i(T)+k_i(T)$, $\lambda(T)=(\lambda_1(T),\ldots,\lambda_n(T)$ and $\fmaj(T)=|\lambda(T)|$.
\begin{prop}\label{dimirrep}
The irreducible representations of $\qc$ are naturally parametrized by pairs $(\mu,\rho)$, where $\mu\in \Fer^*$ and $\rho\in (C_p)_{\mu}$, the stabilizer of any element in the class $\mu$. Moreover the dimension of the irreducible representation indexed by $(\mu,\rho)$ is independent of $\rho$ and it is equal to $|\St_\mu|$.
\end{prop}
\begin{proof}
This is known if $G$ is a reflection group, i.e. if $q=1$ (see \cite{BB} and the references cited there). On the other hand the irreducible representations of $G$ are exactly those of $G(r,p,1,n)$ which are fixed pointwise by $\zeta_q$. By \cite[Theorem 5.3]{St} we know that the multiplicity of the irreducible representation of $G(r,p,1,n)$ indexed by $(\mu,\rho)$ (with $\mu\in \Fer(r,1,p,n)$) in $R^{G(r,p,1,n)}_k$ is given by the number of elements $T\in\St_\mu$ such that $\fmaj(T)=k$. So we deduce that $\mu$ is a representation also for $G$ if and only if $\fmaj(T)\equiv 0 \mod q$ for some (or, equivalently, for all) $T\in \St_\mu$. If $\mu=(\lambda^{(0)},\ldots,\lambda^{(r-1)})$, this  condition is clearly equivalent to $\sum_ii|\lambda^{(i)}|\equiv_q 0$ and the result follows.
\end{proof}
For notational convenience in the next sections, if $\phi$ is an irreducible representation of $G$ indexed by a pair $(\mu,\rho)$ we let $\mu(\phi)\eqdef \mu \in \Fer^*$.

\section{Descent representations}\label{desrep}
In this section we introduce the descent representations of a group $\qc$ which turn out to be a refinement of the homogenous decomposition of the coinvariant algebra. In particular, if $G=G(r,p,q,n)$ and $|\lambda|\equiv 0 \mod q$,  we can consider the submodule $R^G_{\unlhd \lambda}$ of $R^G$ spanned by monomials of total degree $|\lambda|$ and partition degree $\unlhd \lambda$ and we can similarly define $R^G_{\lhd \lambda}$. Following and generalizing \cite{ABR,BC,BB} we denote the quotient module by
$$
R^G_\lambda\eqdef R^G_{\unlhd \lambda}/R^G_{\lhd \lambda}.
$$
We call the modules $R^G_\lambda$ the \emph{descent representations} of $G$.
 A straightforward application of Maschke's theorem implies that, for any $k\equiv 0 \mod q$, we have an isomorphism
$$
\varphi:R^G_k\stackrel{\cong}{\longrightarrow} \bigoplus_{\lambda \vdash k}R^G_{\lambda}
$$
such that every element in $\varphi^{-1}(R^G_{\lambda})$ can be represented by a homogeneous polynomial in $S_q[X]$ of partition degree $\lambda$.
We recall that if $g\in G^*$ then the monomial $a_g$ has partition degree $\lambda(g)$ and so it represents an element in $R^G_{\lambda(g)}$. So the following result is an immediate consequence of Theorem \ref{coba} and Corollary \ref{klo}.
\begin{lem}
Let $\lambda$ be a partition such that $|\lambda|\equiv 0 \mod q$. Then the set
$$
\{a_g:\,g\in G^* \textrm{ and }\lambda(g)=\lambda\}
$$
is a system of representatives of a basis of $R^G_{\lambda}$. In particular $\dim(R^G_{\lambda})=|\{g\in G^{*}:\, \lambda(g)=\lambda \}|$.
\end{lem}

Our next target is an explicit description of the irreducible decomposition of the modules $R^G_{\lambda}$.

\begin{prop}\label{fake}

Let $\phi$ be an irreducible representation of $G=\qc$. Then the multiplicity of $\phi$ in $R^G_{\lambda}$ is given by
$$
\langle \chi_{\phi},\chi_{R^G_{\lambda}}\rangle=|\{T\in \St_{\mu(\phi)}: \, \lambda(T)=\lambda|,$$
where $\mu(\phi)\in \Fer^*$ is defined at the end of \S \ref{irrep}.
\end{prop}
\begin{proof}
This is known for reflection groups $G(r,p,n)$ (see \cite[Theorem 10.5]{BB}). The extension to projective
reflection groups $\qc$ is straightforward by the description of the
irreducible representations of $G$ in \S\ref{irrep}.
\end{proof}
This proposition unifies and generalizes the corresponding coarse results of Lusztig (unpublished) and Kra\'skiewicz-Weyman \cite{KW} in Type A and Stembridge \cite{St} for reflection groups and the corresponding refined results of Adin-Brenti-Roichman \cite{ABR} in Type A and B and of Bagno-Biagioli \cite{BB} for reflection groups.

\section{Tensorial and diagonal invariants}\label{tendia}

In this section we describe the main result of this work (Theorem \ref{main}) which presents an explicit basis for the diagonal invariant algebra of a projective reflection group $G=\qc$ (considered as a free module over the tensorial invariant algebra) in terms of the dual group $G^*$. This result is new also in the generality of reflection groups (see \cite{GG,BL,BB1} for related results in type $A$ and $B$). Here it is really apparent that not only the combinatorics of $G^*$ plays a crucial role in the invariant theory of $G$ as in the previous sections, but also its algebraic structure.\\
Let $S_q[X]^{\otimes k}$ be the $k$-th tensor power of the algebra of polynomials $S_q[X]$ defined in \S\ref{coal}. On this algebra we consider the natural action of the group $G^k$ (where the $i$-th coordinate of $G^k$ acts on the $i$-th factor in $S_q[X]^{\otimes k}$) and of its diagonal subgroup $\Delta G$. We are particularly interested in the corresponding invariant algebras. Every monomial in $S_q[X]^{\otimes k}$ can be described by a $k\times n$-matrix with non negative integer entries such that the sum in each row is divided by $q$. To any such matrix $A$ we associate the monomial $\mathcal X^A\eqdef \prod_{i,j}x_{i,j}^{a_{i,j}}$. Here and in what follows we identify $S_q[X]^{\otimes k}$ with the algebra of polynomials $S_q[\xk]=S_q[x_{i,j}]$ (where $i\in[k]$ and  $j\in[n]$) spanned by monomials whose degree in $x_{i,1},\ldots,x_{i,n}$ is a multiple of $q$ for all $i\in [k]$. The algebra $S_q[\xk]$ is  multigraded by $k$-tuples of partitions with at most $n$ parts: we just say that a monomial $M$ is homogeneous of \emph{multipartition degree} $(\lambda^{(1)},\ldots,\lambda^{(k)})$ if its exponent partition with respect to the variables $x_{i,1},\ldots x_{i,n}$ is $\lambda^{(i)}$ for all $i$. We write in this case $\lambda^{(i)}(M)\eqdef \lambda^{(i)}$ for all $i$ and $\Lambda(M)\eqdef(\lk)$. 

We refer to the algebra $S_q[\xk]^{\Delta G}$ as the \emph{diagonal invariant algebra of $G$}. It is clear that $S_q[\xk]^{\Delta G}$ is generated by the polynomials 
$$(\mathcal X^A)^{\#}\eqdef\frac{1}{|G|}\sum_{g\in \Delta G}g(\mathcal X^A),$$
and that the non zero polynomials of this form are linearly independent.
\begin{lem}\label{colu}
Let $A$ be a $k\times n$ matrix with row sums divided by $q$ and let $s_i$ be the sum of the entries in its $i$-th column. Then $(\mathcal X^A)^{\#}\neq 0$ if and only if the following two conditions are satisfied
\begin{enumerate}
\item $s_i\equiv s_j$ for all $i,j$;
\item $ps_i \equiv 0$ for all $i$.
\end{enumerate}
\end{lem}
\begin{proof}
To prove that $(\mathcal X^A)^{\#}=0$ it is enough to show that there exists a subgroup $H$ of $G$ such that $\sum_{h\in H}h(\mathcal X^A)=0$. Let $h_i=[1;c_1,\ldots,c_r]$, where $c_k=0$ if $k\neq i$ and $c_i=p$ and $H_i$ be the subgroup of $G$ generated by $h_i$. Then
$$
\sum_{h\in H_i}h(\mathcal X^A)=(\zeta_{r/p}^{s_i}+\cdots+\zeta_{r/p}^{rs_i/p})\mathcal X^A=\left\{\begin{array}{ll}\frac{r}{p}\mathcal X^A& \textrm{ if }ps_i\equiv 0\\0& \textrm{otherwise.}
\end{array}\right.
$$
Now let  $h_{i,j}=[1;c_1,\ldots,c_r]$ where $c_k=0$ if $k\neq i,j$, $c_i=1$ and $c_j=-1$. Let $H_{i,j}$ be the subgroup generated by $h_{i,j}$. Then
$$
\sum_{h\in H_{i,j}}h(\mathcal X^A)=(1+\zeta_{r}^{s_i-s_j}+\cdots+\zeta_{r}^{r(s_i-s_j)})\mathcal X^A=\left\{\begin{array}{ll}{r}\mathcal X^A& \textrm{ if }s_i-s_j\equiv 0\\0& \textrm{otherwise.}
\end{array}\right.
$$
To prove the converse let $H$ be the subgroup of $G$ of elements $h$
such that $|h|=1$ and let $A$ be a matrix satisfying (1) and
(2). Since $H$ is generated by the $h_i$'s and the $h_{i,j}$'s, we observe that if $h\in H$ then $h(\mathcal X^A)=c\mathcal X^A$, where $c$ is a positive integer. Then we have
\begin{eqnarray*}
 (\mathcal X^A)^{\#}&=&\frac{1}{|G|}\sum_{\sigma \in \Delta S_n}\sigma \left(\sum_{h\in H} h(\mathcal X^A)\right)\\
 &=&\frac{1}{|G|}\sum_{\sigma \in \Delta S_n}\sigma (\tilde c \mathcal X^A),
\end{eqnarray*}
(where $\tilde c$ is a positive integer), which is clearly non-zero.
\end{proof}
We recall that a \emph{$k$-partite partition} (see \cite{GG,G}) is a $k\times n$ matrix $A=(a_{i,j})$ with non-negative integer entries such that  $a_{i,j}\geq a_{i,j+1}$ whenever $a_{h,j}=a_{h,j+1}$ for all $h<i$.
We denote by $\mathcal B_k(r,p,q,n)$ the set of $k\times n$-matrices which are $k$-partite partitions, with row sums divided by $q$ and column sums satisfying (1) and (2) in Lemma \ref{colu}.
\begin{cor}
The set $\{(\mathcal X^{A})^{\#}: A\in \mathcal B_k(r,p,q,n)\}$ is a basis for the diagonal invariant algebra of $G$. 
\end{cor}
We recall that the algebra $S_q[\xk]^{\Delta G}$, being Cohen-Macaulay (see \cite[Proposition 3.1]{Sta1}), is a free module over the \emph{tensorial invariant algebra} $S_q[\xk]^{G^k}$, and our next target is the description of a basis for $S_q[\xk]^{\Delta G}$ as a free $S_q[\xk]^{G^k}$-module.  
\begin{defn}
Let $\lambda$ be a partition with $n$ parts and $g\in G(r,p,q,n)$. We say that $\lambda$ is $g$-compatible if
 $\lambda-\lambda(g)$ is a partition and
 $g=[|g|;\,\lambda]$.

\end{defn}
We note that in the case of the symmetric group the condition $g=[|g|;\,\lambda]$ in the previous definition is empty and we obtain an equivalent definition of a $\sigma$-compatible partition given in \cite{GG}. The special case of the following result where $G$ is the symmetric group is proved in \cite{GG}.
\begin{thm}\label{bije}
There is a bijection between $\mathcal B_k(r,p,q,n)$ and  $(2k)$-tuples $(g_1,\ldots,g_k;\lk)$ \\ where 
\begin{itemize}
\item $g_1,\ldots,g_k\in G ^*$ are such that $g_1\cdots g_k=1$;
\item $\lambda^{(i)}$ is a $g_i$-compatible partition.
\end{itemize}
The bijection is given by
$$
\Phi(g_1,\ldots,g_k;\lk)=\left( \begin{array}{llll}\lambda^{(1)}_1&\lambda^{(1)}_2&\cdots &\lambda^{(1)}_n\\
                                            \lambda^{(2)}_{\sigma_1(1)}&\lambda^{(2)}_{\sigma_1(2)}&\cdots &\lambda^{(2)}_{\sigma_1(n)}\\
&&&\\
\lambda^{(k)}_{(\sigma_1\ldots\sigma_{k-1})(1)}&\lambda^{(k)}_{(\sigma_1\ldots\sigma_{k-1})(2)}&\cdots&\lambda^{(k)}_{(\sigma_1\ldots\sigma_{k-1})(n)}
\end{array}
\right),
$$
where $\sigma_i=|g_i|$ and the composition of permutations is from left to right.
\end{thm}
\begin{proof}
By \cite[Remark 2.2]{GG},  we know that there exists a bijection $\Phi$ between $2k$-tuples  $(\sigma_1,\ldots,\sigma_k;\lk)$, (where $\lambda^{(i)}$ is $\sigma_i$-compatible for all $i$ and $\sigma_1\cdots \sigma_k=1$) and $k$-partite partitions given by 
\begin{equation}\label{aly}\Phi(\sigma_1,\ldots,\sigma_k;\lk)=\left( \begin{array}{llll}\lambda^{(1)}_1&\lambda^{(1)}_2&\cdots &\lambda^{(1)}_n\\
\lambda^{(2)}_{\sigma_1(1)}&\lambda^{(2)}_{\sigma_1(2)}&\cdots &\lambda^{(2)}_{\sigma_1(n)}\\
&&&\\
\lambda^{(k)}_{(\sigma_1\ldots\sigma_{k-1})(1)}&\lambda^{(k)}_{(\sigma_1\ldots\sigma_{k-1})(2)}&\cdots&\lambda^{(k)}_{(\sigma_1\ldots\sigma_{k-1})(n)}
\end{array}
\right).
\end{equation}
Since $g_i$-compatible implies $\sigma_i$-compatible we deduce that $\Phi(g_1,\ldots,g_k;\lk)$ is a $k$-partite partition. 
Using the fact that $\lambda^{(i)}$ is $g_i$-compatible we compute the sum of the elements in the $i$-th row:
$$
\sum_j \lambda^{(i)}_{\sigma_1\ldots\sigma_{i-1}(j)}=  \sum_j \lambda^{(i)}_{j} \equiv 0 \,\,\mod q,
$$
since $g_i\in G^*$. Now we have to verify that the conditions on the column sums are satisfied. For $j\in [n]$ let $e_j$ be the $1\times n$ matrix whose $i$-th coordinate is 1 and all the others are zero. Consider the elements $\tilde g_i=[\sigma_i,\lambda^{(i)}]\in G(r,p,n)$. Since $\lambda^{(i)}$ is $g_i$-compatible we have that $\tilde g_i$ is a lifting of $g_i$. Then the $j$-th row of the matrix $\tilde g_1\cdots \tilde g_k$ is
\begin{eqnarray}\label{sdf}
e_j(\tilde g_1\cdots \tilde g_k)&=&e_j([\sigma_1;\lambda^{(1)}]\cdots [\sigma_k;\lambda^{(k)}])\\
\nonumber &=&\zeta_r^{\lambda^{(1)}_j}e_{\sigma_1(j)}([\sigma_2;\lambda^{(2)}]\cdots [\sigma_k;\lambda^{(k)}])\\
\nonumber&=&\zeta_r^{\lambda^{(1)}_j+\lambda^{(2)}_{\sigma_1(j)}}e_{(\sigma_1\sigma_2)(j)}([\sigma_3;\lambda^{(3)}]\cdots [\sigma_k;\lambda^{(k)}])\\
\nonumber&=&\zeta_r^{\lambda^{(1)}_j+\lambda^{(2)}_{\sigma_1(j)}+\cdots+\lambda^{(k)}_{\sigma_1\cdots\sigma_{k-1}(j)}}e_j.
\end{eqnarray}
Then, since $g_1\cdots g_k=1$ in $G^*$, we deduce that the exponents of $\zeta_r$ (as $j$ varies in $[n]$) appearing in the previous formula must be all congruent to the same multiple of $r/p$ modulo $r$. Since these exponents are exactly the column sums of the matrix $\Phi(g_1,\ldots,g_k;\lk)$  we conclude that the map $\Phi$ is well-defined.\\
Let's prove that $\Phi$ is surjective. If $A\in \mathcal B_k(r,p,q,n)$ we know that there exist (unique) $(\sigma_1,\ldots,\sigma_k;\lk)$, where $\lambda^{(i)}$ is $\sigma_i$-compatible for all $i$ such that \eqref{aly} is satisfied. We let $g_i=[\sigma_i,\lambda^{(i)}]$. We have to show that $\lambda^{(i)}$ is $g_i$-compatible and for this we only have to verify that $\lambda^{(i)}-\lambda(g_i)$  is a partition. So we have to show that
\begin{equation}\label{hty}
\lambda^{(i)}_j-\lambda^{(i)}_{j+1}\geq \lambda_j(g_i)-\lambda_{j+1}(g_i), 
\end{equation}
for all $j$. 
We observe that the two sides of \eqref{hty} are always congruent modulo $r$. If $j\in \HDes(g_i)$ the right-hand side of \eqref{hty} is $r$. The left-hand side is at least 1 since $\lambda^{(i)}$ is $\sigma_i$-compatible and the result follows.\\
If $j\notin \HDes(g_i)$ then the right-hand side is always in $[0,r-1]$ and the result again follows. To prove that $g_1\cdots g_k=1$ it is enough to read equation \eqref{sdf} backwards.\\ 
Finally we have to prove that $\Phi$ is injective. This is clear from the corresponding fact for the symmetric group and the definition of a $g$-compatible partition.
\end{proof}
We denote by $\Par(r,p,n)$ the set of partitions of length at most $n$ whose parts are all congruent to the same multiple of $r/p$ modulo $r$ and we observe that, if $g\in G^*$, then 
$$
\lambda \textrm{ is $g$-compatible if and only if }\lambda-\lambda(g)\in \Par(r,p,n).
$$
If $g_1,\ldots,g_k\in G^*$ and $g_1\cdots g_k=1$ we let 
$$A(g_1,\ldots,g_k)\eqdef \Phi(g_1,\ldots,g_k;\lambda(g_1),\ldots,\lambda(g_k)).$$

\begin{thm}\label{main}
The set of polynomials
$$
\{(\mathcal X^{A(g_1,\ldots,g_k)})^\# :\,g_1,\ldots,g_k\in G^* \textrm{ and }g_1\cdots g_k=1\},
$$
is a basis for $S_q[\xk]^{\Delta G}$ as a free module over
$S_q[\xk]^{G^k}$ or equivalently for the $\mathbb C$-vector space $(S_q[\xk]/I_+^{G^k})^{\Delta G}$, where $I_+^{G_k}$ is the ideal generated by the non constant homogeneous $G^k$-invariants.
\end{thm}
\begin{proof}
We observe that Theorem \ref{bije} can be restated as follows: if $A\in \mathcal B_k(r,p,q,n)$ then there exist unique $g_1,\ldots,g_k\in  G^*$ with $g_1\cdots g_k=1$ such that 
$$
\mathcal X^A=\mathcal X^{A(g_1,\ldots g_k)}M_1(X_1)\cdots M_k(X_k),
$$
where, for all $i\in [k]$, $M_i$ is a monomial such that
$\lambda(M_i)=\lambda^{(i)}(\mathcal X^A)-\lambda(g_i)$ and
$\lambda(M_i)\in \Par(r,p,n)$. Recall that $\lambda^{(i)}(\mathcal X^A)$ is the
exponent partition of $\mathcal X^A$ with respect to the variables
$x_{i,1},\ldots,x_{i,n}$ (i.e. the partition obtained by reordering
the $i$-th row of $A$). Then with a moment's thought one can deduce that the polynomial $\mathcal
X^{A(g_1,\ldots,g_k)}m_{\lambda(M_1)}(X_1)\cdots
m_{\lambda(M_k)}(X_k)$, where the $m_\lambda$'s are the monomial symmetric functions, is equal to  $\mathcal X^A$ plus a sum of
monomials having multipartition degree strictly smaller than
$\Lambda(\mathcal X^A)$ with respect to the $k$-th cartesian product of the dominance order on partitions.  Now we observe that if
$\lambda\in\Par(r,p,n)$ then $m_\lambda$ is $G$-invariant. For this it is enough to prove that $g(X^{\lambda})=|g|(X^{\lambda})$ for all $g\in G$. By definition of $\Par(r,p,n)$ we know that, for all $i$, $\lambda_i\equiv kr/p \mod r$, for some integer $k$. Then, if $g=[\sigma;c_1,\ldots,c_n]$ we have
$$
g(X^\lambda)=\zeta_r^{c_1\lambda_1+\cdots+c_n\lambda_n}\sigma(X^{\lambda})=\zeta_r^{kr/p(c_1+\cdots+c_n)}\sigma(X^{\lambda})
$$ and the claim follows since $c_1+\cdots+c_n\equiv 0$ $\mod p$. In particular we may conclude that
\begin{equation}\label{straight}
(\mathcal X^A)^\#=(\mathcal X^{A(g_1,\ldots,g_k)})^\#m_{\lambda(M_1)}(X_1)\cdots m_{\lambda(M_k)}(X_k)+F,
\end{equation}
where $\lambda^{(i)}(\mathcal X^A)=\lambda(g_i)+\lambda(M_i)$ and $F\in S_q[\xk]^{\Delta G}_{\lhd \Lambda(\mathcal
  X^A)}$. This implies that the polynomials $(\mathcal
  X^{A(g_1,\ldots,g_k)})^\#$ span the diagonal invariant algebra $S_q[\xk]^{\Delta G}$ as a free module over $S_q[\xk]^{G^k}$.  By Lemma \ref{card} these polynomials have the right cardinality and hence they are a basis.
\end{proof}
\begin{lem}\label{card}
We have
$$
\dim ((S_q[\xk]/I_+^{G^k})^{\Delta G})=|G|^{k-1}
$$
\end{lem}
\begin{proof}
We recall that $S_q[\xk]/I_+^{G^k}\cong \mc G^{\otimes k}$ as
$\Delta G$-modules, by Proposition \ref{reg}. So, if we let $\chi$ be the character of $\mc G^{\otimes k}$ we have $\chi(g)=0$ if $g\neq1$ and $\chi(1)=|G^k|$. So
\begin{eqnarray*}
\dim (\mc G^{\otimes k})^{\Delta G} & = & \frac{1}{|G|}\sum_{g\in G}\chi(g)\\ &=&|G|^{k-1}.
\end{eqnarray*}
\end{proof}
An immediate consequence of Theorem \ref{main} is the following equality
$$
\frac{\Hilb(S_q[\xk]^{\Delta G})(y_1,\ldots,y_k)}{\Hilb(S_q[\xk]^{G^k})(y_1,\ldots,y_k)}=\sum_{\substack{g_1,\ldots,g_k\in G^*:\\g_1\cdots g_k=1}}y_1^{\fmaj(g_1)}\cdots y_k^{\fmaj(g_k)},
$$
where the Hilbert series are considered with respect to the total multidegree.
Theorem \ref{main} and its proof allow us to obtain an important
refinement of the previous identity, considering the multipartition
degree instead of the total multidegree. 

\begin{cor}\label{uou}
We have
\begin {eqnarray*}
\frac{\Hilb(S_q[\xk]^{\Delta G})(Y_1,\ldots,Y_k)}{\Hilb(S_q[\xk]^{G^k})(Y_1,\ldots,Y_k)} &=&\sum_{\substack{g_1,\ldots,g_k\in G^*:\\g_1\cdots g_k=1}}Y_1^{\lambda(g_1)}\cdots Y_k^{\lambda(g_k)},
\end {eqnarray*}
where $Y_i=(y_{i,1},\ldots,y_{i,n})$ and the Hilbert series are
considered with respect to the multipartition degree.
\end{cor}
\begin{proof}
Equation \eqref{straight} implies that 
$$\dim S_q[\xk]^{\Delta G}_{\unlhd \Lambda}= \sum_{\substack{g_1,\ldots,g_k\in G^*:\\g_1\cdots g_k=1}}\sum_{\substack{\Lambda'\in \Par(r,p,n)^k:\\ \Lambda'+\Lambda(g_1,\ldots,g_k)\unlhd \Lambda}}\dim S_q[\xk]^{G^k}_{\Lambda'},$$
and similarly with $\lhd$ instead of $\unlhd$. Therefore
\begin{eqnarray*}
\dim S_q[\xk]^{\Delta G}_{\Lambda} &=& \dim S_q[\xk]^{\Delta G}_{\unlhd \Lambda}-\dim S_q[\xk]^{\Delta G}_{\lhd\Lambda}\\
& = & \sum_{\substack{g_1,\ldots,g_k\in G^*:\\g_1\cdots g_k=1}}\sum_{\substack{\Lambda'\in \Par(r,p,n)^k:\\ \Lambda'+\Lambda(g_1,\ldots,g_k)= \Lambda}}\dim S_q[\xk]^{G^k}_{\Lambda'}.
\end{eqnarray*}
Note that in the last sum there is only one summand corresponding to $\Lambda'=\Lambda-\Lambda(g_1,\ldots,g_k)$ if this is an element in $\Par(r,p,n)^k$, and there are no summands otherwise. For notational convenience, if $\Lambda=(\lk)$ is a multipartition, we denote by $\mathcal Y^\Lambda\eqdef Y_1^{\lambda^{(1)}}\cdots Y_k^{\lambda^{(k)}}=\prod_{i,j}y_{i,j}^{\lambda^{(i)}_j}$. So we have
\begin{eqnarray*}
\Hilb(S_q[\xk]^{\Delta G})& = & \sum_{\Lambda}\dim S_q[\xk]^{\Delta G}_{\Lambda} \mathcal Y^\Lambda\\
& = &  \sum_{\Lambda}\sum_{\substack{g_1,\ldots,g_k\in G^*:\\g_1\cdots g_k=1}}\sum_{\substack{\Lambda'\in \Par(r,p,n)^k:\\ \Lambda'+\Lambda(g_1,\ldots,g_k)= \Lambda}}\dim S_q[\xk]^{G^k}_{\Lambda'}\mathcal Y^{\Lambda}\\
&=&\sum_{\substack{g_1,\ldots,g_k\in G^*:\\g_1\cdots g_k=1}}\sum_{\Lambda'\in \Par(r,p,n)^k}\dim S_q[\xk]^{G^k}_{\Lambda'} \mathcal Y^{\Lambda'+\Lambda(g_1,\ldots,g_k)}\\
&=&\sum_{\substack{g_1,\ldots,g_k\in G^*:\\g_1\cdots g_k=1}}Y_1^{\lambda(g_1)}\cdots Y_k^{\lambda(g_k)}\sum_{\Lambda'\in \Par(r,p,n)^k}\dim S_q[\xk]^{G^k}_{\Lambda'}\mathcal Y^{\Lambda'}\\
&=&\sum_{\substack{g_1,\ldots,g_k\in G^*:\\g_1\cdots g_k=1}} Y_1^{\lambda(g_1)}\cdots Y_k^{\lambda(g_k)} \Hilb (S_q[\xk]^{G^k}).
\end{eqnarray*}
In the last equality we have used the fact that if $F\in
S_q[X]_\lambda$ is a nonzero $G$-invariant polynomial then $\lambda\in
\Par(r,p,n)$. This follows from the fact that the invariants of $\qc$
on $S_q[X]$ coincide with the invariants of $G(r,p,n)$ on $\mc[X]$, and the invariants of $G(r,p,n)$ form a polynomial algebra generated by the $n-1$ independent homogeneous polynomial $e_j(x_1^r,\ldots,x_n^r)$ for $j=1,\ldots,n-1$ together with the monomial $(x_1\cdots x_n)^{r/p}$.
\end{proof}
\section{Kronecker coefficients}\label{krocoe}
We can use the descent representations of a projective reflection
group $G=\qc$ introduced in \S\ref{desrep} and the isomorphism
$$
R^G\cong \bigoplus_\lambda R^G_\lambda
$$
to give to the coinvariant algebra the structure of a partition-graded module. By means of this grading of the coinvariant algebra we can also decompose the algebra
 $$\frac{S_q[\xk]}{I^{G^k}_+}\cong \underbrace{R^G\otimes\cdots\otimes R^G}_k$$ and its diagonal invariant subalgebra
$$
\left( \frac{S_q[\xk]}{I^{G^k}_+}\right)^{\Delta G}\cong \frac{S_q[\xk]^{\Delta G}}{J^{G^k}_+}
$$
in homogeneous components whose degrees are $k$-tuples of partitions with at most $n$ parts. Here  $J^{G^k}_+$ is the ideal generated by homogeneous $G^k$-invariant polynomials of positive degree inside  $S_q[\xk]^{\Delta G}$.
The last isomorphism is due to the fact that the natural map
$$
S_q[\xk]^{\Delta G}/J^{G^k}_+ \rightarrow (S_q[\xk]/I^{G^k}_+)^{\Delta G}
$$
is an isomorphism since the averaging operator $F\mapsto F^\#$ provides an inverse.
In particular we say that an element in $S_q[\xk]/I^{G^k}_+$ is homogeneous of \emph{multipartition degree} $(\lk)$ if it belongs to
$R^G_{\lambda^{(1)}}\otimes \cdots \otimes R^G_{\lambda^{(k)}}$ by means of the above mentioned isomorphism. 

We define the \emph{refined fake degree polynomial} $f^{\phi} (y_1,\ldots, y_n)$  of  $G$ as the polynomial whose coefficient of $y_1^{\lambda_1}\cdots y_n^{\lambda_n}$ is the multiplicity of the irreducible representation $\phi$ of $G$ in $R^G_\lambda$. If $\phi_1,\ldots,\phi_k$ are $k$ irreducible  representations of $G$ we define the \emph{Kronecker coefficients} of $G$ by
$$
g_{\phi_1,\ldots,\phi_k}\eqdef\frac{1}{|G|}\sum_{g\in G}\chi^{\phi_1}(g)\cdots \chi^{\phi_k}(g).
$$
One may note that the knowledge of the Kronecker coefficients is equivalent to the knowledge of the irreducible decomposition of the tensor product of $k-1$ irreducible representations of $G$.
If $G=G(r,p,q,n)$ and $\mu_1,\ldots, \mu_k\in \Fer^*=\Fer(r,q,p,n)$, we define the \emph{coarse Kronecker coefficients} of $G$ by
$$g_{\mu_1,\ldots, \mu_k}\eqdef \sum_{\substack{\phi_1,\ldots,\phi_k \in \Irr(G):\\ \mu(\phi_i)=\mu_i \forall i}}g_{\phi_1,\ldots,\phi_k},$$
where $\Irr(G)$ denotes the set of all irreducibel repesentations of $G$. 
Note that if $p=1$, for example for the wreath products $G(r,n)$, then the coarse Kronecker coefficients necessarily coincide with the standard ones.
The following result is a consequence of \cite[Theorem 5.11]{So} (see also \cite[Theorem 3.1]{Ca} for the corresponding result for the symmetric group).
\begin{thm}\label{fcgen}We have
$$
\Hilb\Big(\frac{S_q[\xk]^{\Delta G}}{ J^{G^k}_+}\Big)(Y_1,\ldots,Y_k)=\sum_{\phi_1,\ldots,\phi_k\in \Irr (G)}g_{\phi_1,\ldots,\phi_k}f^{\phi_1}(Y_1)\cdots f^{\phi_k}(Y_k)
$$
where the sum is taken over all $k$-tuples of irreducible representations of $G$.
\end{thm}
By Proposition \ref{fake} we deduce that 
$$
f^{\phi}(Y)=\sum_{T\in \St_{\mu(\phi)}} Y^{\lambda(T)}.
$$
Then 
Theorem \ref{fcgen} can be restated as follows
\begin{equation}\label{ghj}
\Hilb\Big(\frac{S_q[\xk]^{\Delta G}}{ J^{G^k}_+}\Big)(Y_1,\ldots,Y_k)=\sum_{T_1,\ldots,T_k\in \St^*}g_{\mu(T_1),\ldots,\mu(T_k)}Y_1^{\lambda(T_1)}\cdots Y_k^{\lambda(T_k)},
 \end{equation}
where $\St^*\eqdef\St(r,q,p,n)$. 
The following result relates the multivariate descents statistics considered so far on the group $\qc$ with the corresponding statistics on tableaux in $\St(r,p,q,n)$, by means of the coarse Kronecker coefficients.
\begin{cor}\label{maincomb}
Let $G=G(r,p,q,n)$ and $\St=\St(r,p,q,n)$. Then
$$
\sum_{\substack{g_1,\ldots,g_k\in G\\g_1\cdots g_k=1}}Y_1^{\lambda(g_1)}\cdots Y_k^{\lambda(g_k)}=\sum_{T_1,\ldots,T_k\in \St}g_{\mu(T_1),\ldots,\mu(T_k)}Y_1^{\lambda(T_1)}\cdots Y_k^{\lambda(T_k)}.
$$ 
\end{cor}
\begin{proof}
This follows from Theorem \ref{main}, Corollary \ref{uou} and Equation \eqref{ghj} (with $G=G^*$) since the set of polynomials $\{(\mathcal X^{A(g_1,\ldots,g_k)})^\#\}$ represents also a homogeneous basis with respect to the multipartition degree of $S_q[X_1,\ldots,X_k]^{\Delta G}/ J^{G^k}_+$.
\end{proof}

We may observe that Corollary \ref{maincomb} can also be seen as a purely combinatorial algorithm to compute the coarse Kronecker coefficients of $G$. This can be achieved in a way which is similar to the corresponding result for the symmetric group (see \cite[\S4]{Ca}).
One further consequence is the existence of a generalized multivariate Robinson-Schensted correspondence for projective reflection groups $\qc$.
\begin{cor}\label{rs}
There exists a correspondence $\mathcal T$ that associates to any $k$-tuple $(g_1,\ldots,g_k)$ of elements in $G$ such that $g_1\cdots g_k=1$ a $k$-tuple of tableaux $T_1,\ldots,T_k$ such that
\begin{itemize}
\item $|\mathcal T^{-1}(T_1,\ldots,T_k)|=g_{\mu(T_1),\ldots,\mu(T_k)}$;
\item $\lambda(g_i)=\lambda(T_i)$ for all $i$.
\end{itemize}
\end{cor}

In the next section we describe a bijective proof of Corollary \ref{maincomb} in the case $k=2$.
\section{The Robinson-Schensted correspondence}\label{robshe}
The Robinson-Schensted correspondence is a very classical tool in the study of the algebraic combinatorics of the symmetric groups and it has a rather natural generalization to the wreath product groups $G(r,n)$ that we describe later. On the other hand there is no such correspondence for the other reflection groups of the form $G(r,p,n)$, probably because such correspondence should involve the representation theory of the dual group. In this section we extend this correspondence to all projective reflection groups of the form $G(r,p,q,n)$.

Recall the classical Robinson-Schensted correspondence  from \cite[\S 7.11]{Sta}). 
This correspondence has been extended to wreath product groups $G(r,n)$ in \cite{SW} in the following way. Given $g\in G(r,n)$ and  $j\in[0,r-1]$, we let $\{i_1,\ldots,i_h\}=\{l\in [n]: c_l(g)=j\}$ and we consider the two-line array $A_j=\left(\begin{array}{cccc}i_1&i_2&\cdots&i_h\\ \sigma(i_1)&\sigma(i_2)&\cdots&\sigma (i_h)\end{array}\right)$, where $\sigma=|g|$, and the pair of tableaux $(P_j,Q_j)$ obtained by applying the Robinson-Schensted correspondence to $A_j$. Then the Stanton-White correspondence 
$$
g\mapsto (P(g),Q(g))\eqdef \Big((P_0,\ldots,P_{r-1}),(Q_0,\ldots,Q_{r-1})\Big)
$$
is a bijection between $G(r,n)$ and pairs of tableaux of the same shape in $\St(r,1,n)$. Furthermore we have $\lambda(g)=\lambda(Q(g))$ and $\lambda(\bar g^{-1})=\lambda(P(g))$.\\
Now let $g\in G(r,p,q,n)$ and $\tilde g\in G(r,p,n)$ be a lifting of $g$. Then the classes in $\St(r,p,q,n)$ of the tableaux $P(\tilde g)$ and $Q(\tilde g)$ obtained by applying the previous correspondence depend uniquely on $g$ and not on the lifting $\tilde g$. Therefore one can define a map $g\mapsto(P(g),Q(g))$  which associates to any element in $G(r,p,q,n)$ a pair of tableaux in $\St(r,p,q,n)$ of the same shape. The following result 
is the natural generalization of the Stanton-White correspondence to projective reflection groups.
\begin{thm}\label{projRS}
Let $P,Q$ be two tableaux in $\St(r,p,q,n)$ of the same shape $\mu$. Then 
$$|\{g\in \qc:P(g)=P \textrm{ and }Q(g)=Q\}|=|(C_q)_\mu|,$$
where $(C_q)_\mu$ is the stabilizer in $C_q$ of any element in the class $\mu$.
\end{thm}
\begin{proof}
We first observe that the map  $g\mapsto (P(g),Q(g))$ is a bijection between $G(r,p,n)$ and pairs of tableaux in $\St(r,p,n)$ with the same shape, by restriction of the corresponding result for wreath products. The number of pairs of tableaux with the same shape in $\St(r,p,n)$ which represent $P$ and $Q$ is exactly $q\cdot|(C_q)_\mu|$. In fact, we can choose a lifting for $P$ in $q$ ways and for any such choice the lifting of $Q$ can be chosen in $|(C_q)_\mu|$ (since we require the two liftings to have the same shape). The result follows since the number of liftings in $G(r,p,n)$ of a given element in $\qc$ is $q$.
\end{proof}
We observe that Theorem \ref{projRS} provides a bijective proof that
$$
|G|=\sum_{\phi\in \Irr(G^*)}(\dim \phi)^2,
$$
since $\dim \phi=|\St_{\mu(\phi)}|$ and, given $\mu\in \Fer$, we have $|\{\phi\in \Irr(G^*):\mu(\phi)=\mu\}|=|(C_q)_\mu|$.

\section{Galois automorphisms}\label{gal}
The final target of this work is to use some of the theory developed in the previous sections to solve a problem posed in \cite[Question 6.3]{BRS}. The objects of our study here are again Hilbert series of invariant algebras as in \S \ref{tendia} but with a new ingredient given by a Galois automorphism.\\ 
It is known that (any representation of) the reflection group $G(r,p,n)$ is defined over $\mathbb Q(\zeta_r)$ (see, e.g., \cite[\S6]{St}) and in particular this is true also for the projective reflection group $\qc$. Now, if $\sigma\in \Gal(\mathbb Q[e^{\frac{2\pi i}{r}}], \mathbb Q)$, we have $\sigma (C_q)=C_q$ and so we can apply $\sigma$ on $\qc$ (just letting $\sigma$ act entrywise on the matrices representing the elements in $G$). Moreover we observe that, since $\sigma(\zeta_r)=\zeta_r^d$ for some $d$ such that $GCD(r,d)=1$, we also have $g^\sigma\eqdef \sigma(g)\in G$ for all $g\in G$, i.e. $\sigma\in \mathrm{Aut}(G)$. 

The setting is similar to that of \S \ref{tendia} with $k=2$: we consider the following twisted diagonal subgroup of $G\times G$
$$
\Delta^\sigma G\eqdef\{(g,g^\sigma):\,g\in G\}.
$$
We recall that $G \times G$ acts on the symmetric algebra $S_q[X_1,X_2]$ and that this algebra has a bipartition degree given by the exponent partitions in the two sets of variables, i.e.
$$S_q[X_1,X_2]=\bigoplus_{\lambda^{(1)},\lambda^{(2)}}\mathbb S_q[X_1,X_2]_{\lambda^{(1)},\lambda^{(2)}},$$ where the sum is over all pairs of partitions of size multiples of $q$. 

The coinvariant algebra of $R^{G\times G}$ is canonically isomorphic to $R^G\otimes R^{G}$
 and so it also affords a bipartition degree given by
$$
R^{G\times G}_{\lambda^{(1)},\lambda^{(2)}}\cong R^G_{\lambda^{(1)}}\otimes R^{G}_{\lambda^{(2)}}.
$$
We are interested in the subalgebra of $R^{G\times G}$ consisting of $\Delta^\sigma G$-invariants and in particular to its Hilbert series with respect to the bipartition degree defined above.

The following result was proved in \cite{BRS} for all reflection groups in its unrefined version (i.e. considering only the bidegree in $\mathbb N^2$ and not the bipartition degree).
\begin{thm}\label{gsig}
Let $G=\qc$. Then
$$
G^\sigma(Y_1,Y_2)\eqdef\Hilb\left( \frac{S_q[X_1,X_2]^{\Delta^\sigma G}}{J_+^{G\times G}}\right)(Y_1,Y_2)=\sum_{\phi \in \Irr(G)}f^{\sigma \phi}(Y_1)f^{\bar \phi}(Y_2).
$$
\end{thm}
\begin{proof}
We let $\chi_{R^G}(Y)$ be the partition-graded character of $R^G$, i.e. $\chi_{R^G}(Y)=\sum_{\lambda} Y^\lambda \chi_{R_\lambda ^G}$. The definition of the polynomials $f^\phi (Y)$ provides also
$$
\chi_{R^G}(Y)=\sum_{\phi\in \Irr(G)} f^\phi(Y) \chi_\phi.
$$
So we may deduce that

\begin{eqnarray*}
 \Hilb\left( \frac{S_q[X_1,X_2]^{\Delta^\sigma G}}{J_+^{G\times G}}\right)(Y_1,Y_2)&=& \Hilb \left( \Big(\frac{S_q[X_1,X_2]}{I_+^{G\times G}}\Big)^{\Delta^\sigma G}\right) (Y_1,Y_2)\\
&=&\Hilb \left((R^G\otimes R^G)^{\Delta^\sigma G}) \right)(Y_1,Y_2)\\
&=&\sum_{\lambda^{(1)},\lambda^{(2)}} \dim \left( (R^G_{\lambda^{(1)}}\otimes R^G_{\lambda^{(2)}})^{\Delta^\sigma G} \right) Y_1^{\lambda^{(1)}}Y_2^{\lambda^{(2)}}\\
&=& \sum_{\lambda^{(1)},\lambda^{(2)}} \langle \chi_{R^G_{\lambda^{(1)}}\otimes R^G_{\lambda^{(2)}}},1\rangle_{\Delta^\sigma G}Y_1^{\lambda^{(1)}}Y_2^{\lambda^{(2)}}\\
&=& \sum_{\lambda^{(1)},\lambda^{(2)}}\Big( \frac{1}{|G|}\sum_{g\in G}\chi_{R^G_{\lambda^{(1)}}}(g) \chi_{R^G_{\lambda^{(2)}}}( g^\sigma)\Big)Y_1^{\lambda^{(1)}}Y_2^{\lambda^{(2)}}\\
&=& \frac{1}{|G|}\sum_{g\in G}\chi_{R^G}(Y_1)(g)\cdot \chi_{R^G}(Y_2)(g^\sigma)\\
&=& \frac{1}{|G|}\sum_{g\in G} \sum_{\phi_1\in \Irr(G)}f^{\phi_1}(Y_1)\chi_{\phi_1}(g) \sum_{\phi_2\in \Irr(G)}f^{\phi_2}(Y_2)\chi_{\phi_2}(g^\sigma)\\
&=& \sum_{\phi_1,\phi_2\in \Irr(G)}f^{\phi_1}(Y_1)f^{\phi_2}(Y_2)\frac{1}{|G|}\sum_{g\in G}\chi_{\phi_1}(g)\chi_{\sigma \phi_2}(g)\\
&=& \sum_{\phi \in \Irr(G)} f^{\sigma \phi}(Y_1)f^{\bar \phi}(Y_2).
\end{eqnarray*}

\end{proof}

Now we observe that $\sigma(\zeta_r)=\zeta_r^j$ for some $j$ such that $\GCD(j,r)=1$. Then, if  

$\mu=(\lambda^{(0)},\ldots,\lambda^{(r-1)})\in \Fer^*$, we let $\sigma \mu=(\lambda^{(j\cdot 0)}, \lambda^{(j\cdot 1)},\ldots, \lambda^{(j\cdot (r-1))})$ where the indices have to be considered modulo $r$. One can easily verify that $\sigma \mu$ is a well-defined element in $\Fer^*$ and we have that if $\phi\in \Irr(G)$ then $\sigma(\mu(\phi))=\mu(\sigma\phi)$ (see \cite[\S6]{BRS} for related results in $G(r,n)$).\\
As an application of Theorem \ref{projRS}, we show that the polynomial $G^\sigma(Y_1,Y_2)$  takes the following simple form in terms of the dual group $G^*$. This result is a refinement and a generalization to the projective reflection group $\qc$ of \cite[Theorem 1.3]{BRS}.
\begin{cor}
 For any projective reflection group $G=G(r,p,q,n)$ and any Galois automorphism $\sigma\in \Gal(\mathbb Q[e^{\frac{2\pi i}{r}}]/\mathbb Q)$ we have
$$G^\sigma(Y_1,Y_2)=\sum_{g\in G^*}Y_1^{\lambda(g^\sigma)}Y_2^{\lambda(g^{-1})}.$$
\end{cor}

\begin{proof}By Theorem \ref{projRS} and Proposition \ref{fake} we have
\begin{eqnarray*}
\sum_{g\in G^*}Y_1^{\lambda(g^\sigma)}Y_2^{\lambda(g^{-1})} & = &\sum_{\mu\in\Fer^*}|(C_q)_{\mu}|\sum_{P,Q\in\St_\mu}Y_1^{\lambda(Q^\sigma)}Y_2^{\lambda(\bar P)}\\
&=& \sum_{\phi \in \Irr(G)} f^{\sigma \phi}(Y_1)f^{\bar \phi}(Y_2)
\end{eqnarray*}
\end{proof}
The unrefined version of the previous corollary 
$$
G^{\sigma}(y_1,y_2)=\sum_{g\in G^*}y_1^{\fmaj(g^{\sigma})}y_2^{\fmaj(g^{-1})}
$$
provides an answer to \cite[Question 6.3]{BRS}.

\emph{E-mail address: }{\tt caselli@dm.unibo.it}
\end{document}